\documentclass[12pt]{article}

\usepackage{setspace}
\usepackage{amssymb,amsmath,amsthm}
\usepackage{calc}
\usepackage{verbatim}
\usepackage{cite}
\usepackage{epsfig}
\usepackage{graphicx}
\usepackage{graphics}
\usepackage{xcolor}\usepackage[hmargin=3cm,vmargin=3.5cm]{geometry}
\usepackage{chngcntr}
\counterwithin{figure}{section}

\newtheorem{defn}{Definition}[section]

\newtheorem{prop}[defn]{Proposition}
\newtheorem{theorem}[defn]{Theorem}
\newtheorem{cor}[defn]{Corollary}

\newtheorem{conj}[defn]{Conjecture}

\newcommand{\real}{\mathbb{R}}
\newcommand{\Complex}{\mathbb{C}}

                                {\null\hfill$\Box$\par\medskip\medskip\medskip\medskip}

\makeatletter
\renewcommand{\pod}[1]{\mathchoice
  {\allowbreak \if@display \mkern 18mu\else \mkern 8mu\fi (#1)}
  {\allowbreak \if@display \mkern 18mu\else \mkern 8mu\fi (#1)}
  {\mkern4mu(#1)}
  {\mkern4mu(#1)}
}

\begin{document}

\title{On the Unimodality of Independence Polynomials \\ of Very Well-Covered Graphs}
\author{J.~I.~Brown\footnote{Communicating author.  Jason.Brown@dal.ca}~ and B. Cameron\\
Department of Mathematics and Statistics \\
Dalhousie University, Halifax, NS B3H 3J5, Canada\\
\date{}
}
\maketitle


\begin{abstract}
The independence polynomial $i(G,x)$ of a graph $G$ is the generating function of the numbers of independent sets of each size. A graph of order $n$ is very well-covered if every maximal independent set has size $n/2$. Levit and Mandrescu conjectured that the independence polynomial of every very well-covered graph is unimodal (that is, the sequence of coefficients is nondecreasing, then nonincreasing). In this article we show that every graph is embeddable as an induced subgraph of a very well-covered graph whose independence polynomial is unimodal, by considering the location of the roots of such polynomials.
\end{abstract}
\maketitle

\section{Introduction}
A subset $S$ of the vertex set of a (finite, undirected) graph $G$  is said to be {\it independent} if $S$ induces a graph with no edges. The {\it independence polynomial} of a graph $G$ is defined to be 
\[ \displaystyle{i(G,x)=\sum_{k=0}^{\alpha}i_kx^k},\]
where $i_k$ is the number of independent sets of size $k$ in $G$ and $\alpha = \alpha(G)$, the {\it independence number} of $G$, is the size of the largest independent set in $G$. 
The independence polynomial is the generating function of the {\it independence sequence} $\langle i_0,i_1,\ldots,i_\alpha \rangle$. The independence polynomial of a graph has been of considerable interest \cite{INDPOLY, Arocha1984,Andrews2000,INDFIRST,Gutman1992,Hoede1994,Levit2002,VWCOVEREDLOGCONCAVE,Gutman1995} since it was first defined by Gutman and Harary in 1983 as a generalization of the matching polynomial. 

For many graph polynomials (such as matching \cite{HEILMANN}, chromatic \cite{read,chromhuh} and reliability \cite{colbook,relhuh} polynomials), the (absolute value of the) coefficient sequence, under a variety of bases expansions, have long been conjectured to be (or proven to be) {\it unimodal}, that is, nondecreasing then nonincreasing. We say that a polynomial is unimodal if its sequence of coefficients is unimodal.

What can we say about the unimodality of independence polynomials? They certainly form a sequence of positive integers. Alavi et al. \cite{ALAVI} showed, in general, that the independence sequence $\langle i_{k} \rangle$  of a graph $G$ can be far from unimodal, for example, the graph $K_{25}+4K_{2}$ has independence sequence $\langle 1,33,24,32,16\rangle$. More examples of graphs with nonunimodal independence sequences can be found in \cite{ALAVI}. 

However, there are classes of graphs for which the independence coefficients are indeed unimodal. In a beautiful paper \cite{Chudnovsky2007}, Chudnovsky and Seymour provedthat the coefficients of the independence polynomials of {\it claw--free} graphs (that is, those without an induced star on $4$ vertices) are unimodal. 

Another highly structured family of graphs with respect to independence are {\it well-covered} graphs, those whose maximal independent sets all have the same size (complete graphs and the $5$-cycle are examples). The structure of such graphs has attracted considerable attention in the literature, with characterizations for those of high girth \cite{nowawcgirth}. In \cite{brownwc}, the authors conjectured that the independence coefficients of well-covered graphs were unimodal, and showed that every graph $G$ can be embedded as an induced subgraph of such a well-covered graph. However, Michael and Traves \cite{TRAVES} later disproved the conjecture. A conjecture due to Alavi et al. \cite{ALAVI} that is still open is that the independence polynomial of a tree is unimodal.

Finally, Levit and Mandrescu \cite{levitgeneric} amended the original unimodality conjecture on well-covered graphs as follows. A {\it very well-covered} graph $G$ of order $n$ (that is, on $n$ vertices) is a well-covered graph for which every maximal independent set has size $n/2$; for example, the complete bipartite graphs $K_{m,m}$ are very well-covered. Other examples are afforded by the following construction. Let $G$ be any graph. Form $G^\ast$, the {\it leafy extension} of $G$ (sometimes also called the {\it corona} of $G$ with $K_{1}$)  from $G$ by attaching, for each vertex $v$ of $G$ a new vertex $v^\ast$ to $v$ with an edge (such a vertex is called a {\it pendant} vertex); leafy extensions are always very well-covered (more about that shortly). 

Levit and Mandrescu conjectured that the coefficients of the independence polynomials of a very well-covered graph are unimodal, and to date, the conjecture remains open. Some partial results have been proven on the tail of independence sequences of very well-covered graphs  \cite{LEVITBOOK} and the first $\lceil\tfrac{\alpha}{2}\rceil$ terms have been shown to be nondecreasing for well-covered graphs \cite{TRAVES}. The conjecture is known to hold when $\alpha(G) \leq 9$ \cite{LEVITBOOK} and for leafy extensions of any graph $G$ where $\alpha(G) \leq 8$ \cite{Chen}, or where $G$ is a path or star \cite{VWCOVEREDLOGCONCAVE}. In this paper we shall show that Levit and Mandrescu's conjecture holds for some iterated leafy extensions of {\it any} graph $G$.

\section{Unimodality of Independence Polynomials of Leafy Extensions and Sectors in the Complex Plane}

The leafy extension $G^{\ast}$ of any graph $G = (V,E)$ of order $n$ is always very-well-covered. Clearly,  $\alpha(G^{\ast}) \leq n$, as the graph has a perfect matching (and no independent set can contain two vertices that are matched). Moreover, $\alpha(G^{\ast}) =  n$ as any independent set $I$ of $G$ can be extended to one in $G^{\ast}$ by adding in any subset of $(V-I)^{\ast} = \{v^{\ast}: v \in V-I\}$. It follows (see also \cite{VWCOVEREDLOGCONCAVE}) that if 
$i(G,x) = \sum i_{k}x^{k}$, then
\begin{eqnarray} 
i(G^{\ast},x) & = & \sum i_{k} x^{k}(1+x)^{n-k} \nonumber \\
                   & = & (1+x)^n \cdot i \left( G,\frac{x}{1+x} \right) . \label{leafyformula}
\end{eqnarray}

For a graph $G$ and positive integer $k$, let $G^{k\ast}$ denote the {\it $k$--th iterated leafy extension} of $G$, that is, the graph formed by recursively attaching pendant vertices, $k$ times:
\[ G^{k\ast} = \left\{ \begin{array}{ll}
                               G^\ast & \mbox{if $k = 1$},\\
                               (G^{(k-1)\ast})^\ast & \mbox{if $k \geq 2$.}
                              \end{array}  
                     \right. 
\]                    
Figure~\ref{leafyexample} shows the graph $P_{4}^{2\ast}$.

\begin{figure}[htp]
\begin{center}
\includegraphics[scale=0.5]{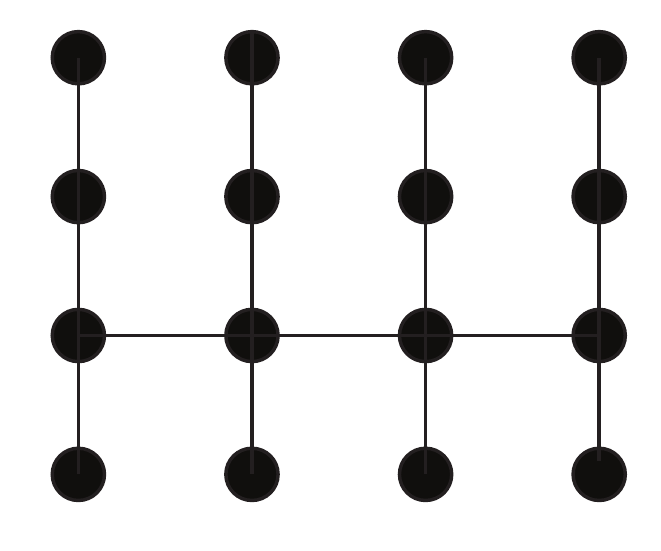}
\caption{The double leafy extension of $P_{4}$.}
\label{leafyexample}
\end{center}
\end{figure}

We can extending  formula (\ref{leafyformula}) to higher iterations of the $\ast$ operation as follows. 

\begin{prop}\label{iteratedleafyprop} For any graph $G$ of order $n$ and any positive integer $k$,
$$i(G^{k\ast},x)=i(G,\tfrac{x}{kx+1})(kx+1)^n\prod_{\ell=1}^{k-1}\left(\ell x+1\right)^{n2^{k-\ell-1}}.$$
\end{prop}
\begin{proof}
We proceed by induction on $k$, the number of iterations of the $*$ operation.  The base case follows directly from (\ref{leafyformula}), so we can assume that the result holds for some $k\ge1$, i.e.,
$$i(G^{k\ast},x)=i(G,\tfrac{x}{kx+1})(kx+1)^n\prod_{\ell=1}^{k-1}\left(\ell x+1\right)^{n2^{k-\ell-1}}.$$
A trivial induction shows that $G^{k\ast}$ has order $n2^{k}$.
From these, the fact that $G^{(k+1)\ast} = (G^{k\ast})^\ast$ and formula (\ref{leafyformula}) we derive that
\begin{eqnarray*}
i(G^{(k+1)*},x)&=&(1+x)^{n2^{k}}i(G^{k*},\tfrac{x}{x+1})\\
&=&\left(1+x\right)^{n2^{k}}i\left(G,\tfrac{\tfrac{x}{x+1}}{\tfrac{kx}{x+1}+1}\right) \left(k\left(\tfrac{x}{x+1}\right)+1\right)^n\prod_{\ell=1}^{k-1}\left(\ell \left(\tfrac{x}{x+1}\right)+1\right)^{n 2^{k-\ell-1}}\\
&=&\left(1+x\right)^{n2^{k}}i\left(G,\tfrac{x}{(k+1)x+1}\right)\left(\tfrac{(k+1)x+1}{x+1}\right)^n\prod_{\ell=1}^{k-1}\left(\tfrac{(\ell+1)x+1}{x+1}\right)^{n2^{k-\ell-1}}\\
&=&\frac{\left(1+x\right)^{n2^{k}}}{\left(1+x\right)^{n2^{k-1}}}i\left(G,\tfrac{x}{(k+1)x+1}\right)\left((k+1)x+1\right)^n\prod_{\ell=1}^{k-1}\left((\ell+1)x+1\right)^{n2^{k-\ell-1}}\\
&=&(1+x)^{n2^{k-1}} i\left(G,\tfrac{x}{(k+1)x+1}\right)\left((k+1)x+1\right)^n\prod_{\ell=1}^{k-1}\left((\ell+1)x+1\right)^{n2^{k-\ell-1}}\\
&=&i\left(G,\tfrac{x}{(k+1)x+1}\right)\left((k+1)x+1\right)^n\prod_{\ell=0}^{k-1}\left((\ell+1)x+1\right)^{n2^{k-\ell-1}}\\
&=&i\left(G,\tfrac{x}{(k+1)x+1}\right)\left((k+1)x+1\right)^n\prod_{\ell=1}^{k}\left(\ell x+1\right)^{n2^{(k+1)-\ell-1}}.\\
\end{eqnarray*}    
\end{proof}

There are many techniques for proving that a sequence is unimodal (see for example, \cite{STANELYSURVEY} and \cite{BRENTISURVEY}). One that has been frequently applied is due to Newton (c.f. \cite[pp. 270-271]{COMTET}), who proved that if a polynomial $p(x) = a_{0} + a_{1}x + \cdots + a_{n}x^{n}$ with positive coefficients has all real roots, then the sequence $\langle a_{0},a_{1},\ldots,a_{n} \rangle$ 
satisfies 
\[ a_{i}^{2} \geq \frac{i+1}{i}\frac{n-i+1}{n-i} a_{i-1}a_{i+1},\]
and hence is {\it log concave}, that is $a_{i}^{2} \geq a_{i-1}a_{i+1}$ for all $0 < i < n$ (in fact, the sequence is {\it strictly} log concave as $a_{i}^{2} > \frac{i+1}{i}\frac{n-i+1}{n-i} a_{i-1}a_{i+1}$ holds for all relevant $i$). In such a case, it follows directly that the sequence is unimodal as well. Newton's simple but elegant theorem has been used to prove that a variety of sequences (and polynomials) are unimodal, such as matching polynomials \cite{HEILMANN} and the independence polynomials of claw-free graphs \cite{Chudnovsky2007}. 

From Proposition~\ref{iteratedleafyprop} we see at once that all independence roots of $G$ are real if and only if the same is true of its leafy extension. As most independence polynomials have a non-real root \cite{BROWNAVERAGE}, we won't be able to get at our desired result, namely that for {\it any} graph $G$ the independence polynomial of some iterated leafy extension of $G$ is unimodal, via Newton's theorem. 

However, Newton's theorem is only a sufficient condition for the coefficient sequence to be log concave. Brenti et al. \cite{SLCREGION} weakened the conditions as follows:

\begin{prop}[\cite{SLCREGION}]
\label{propSLC}
If all the roots $z$ of the polynomial $f(x)\in\real[x]$ are in the region 
\[ \{z\in\Complex:\vert\text{arg}(z)\vert<\tfrac{\pi}{3}\},\] then the sequence of coefficients of $f(x)$ is strictly log concave and alternates in sign.
\end{prop}

Replacing $f(x)$ by $f(-x)$, we derive that:
 
\begin{cor}\label{corSLC}
If all the roots $z$ of the polynomial $f(x)\in\real[x]$ are in the region 
\[ \{z\in\Complex:\tfrac{2\pi}{3}<\vert\text{arg}(z)\vert<\tfrac{4\pi}{3}\},\] 
then the sequence of coefficients of $f(x)$ is strictly log concave (and the sequence of coefficients of $f(x)$ is either all positive or all negative).
\end{cor}

We shall make good use out of this corollary now to prove our main result on log concavity of independence polynomials of leafy extensions, via an excursion through their roots.

\begin{theorem}\label{thmexistsk}
For all graphs $G$ let 
\[ M = \max \left\{ \frac{\frac{1}{\sqrt{3}} |\operatorname{Im} z | + | \operatorname{Re} z|}{|z|^2} : z \mbox{ is a root of } I(G,x)\right\} . \] 
If $k>M$, then the coefficients of $I(G^{k\ast},x)$ are strictly log concave.  
\end{theorem}
\begin{proof}
From Proposition~\ref{iteratedleafyprop} it follows that if $r_1,\ldots,r_m$ are the roots of $I(G,x)$, then the roots of $I(G^{k*},x)$ are $\tfrac{r_i}{1-kr_i}$ for $i=1,2,\ldots,m$ along with the rational numbers $\tfrac{-1}{\ell}$ for $\ell=1,2,\ldots,k$.
Let $r$ be any root of $I(G,x)$, and set $a = \operatorname{Re} r$ and $b = \operatorname{Im} r$. Note that either $a$ or $b$ is nonzero since $0$ is not the root of any independence polynomial and likewise, $r \neq 1/k$ for all $k\ge 0$. We expand a root of $I(G^{k*},x)$ to obtain,
\begin{eqnarray}
\frac{r}{1-kr}&=& \frac{a+ib}{1-k(a+ib)}\nonumber \\
&=& \frac{a+ib}{(1-ka)-ikb}\cdot\frac{(1-ka)+ikb}{(1-ka)+ikb}\nonumber \\
&=& \frac{a(1-ka)+iakb+ib(1-ka)-kb^2}{(1-ka)^2+k^2b^2}\nonumber \\
&=& \frac{a(1-ka)-kb^2+i(akb+b(1-ka))}{(1-ka)^2+k^2b^2}\nonumber \\
&=& \frac{(a-ka^2-kb^2)+ib}{(1-ka)^2+k^2b^2}.
\end{eqnarray}

We now wish to show that for sufficiently large $k$, the root $z = \frac{r}{1-kr}$ of $I(G^{k\ast},x)$ lies in the sector $\{ z\in\mathbb{C}:\tfrac{2\pi}{3}<|\text{arg}(z)|< \tfrac{4\pi}{3}\}$; the result will then follow immediately from Corollary \ref{corSLC} (as the negative rational roots obviously lie in the sector).  It is clear to see that $z$ lies in the sector if and only if $\operatorname{Re} z<0$ and $\vert \tfrac{\operatorname{Im} z}{\operatorname{Re} z}\vert < \sqrt{3}$. Now, $\operatorname{Re} z = \tfrac{a-ka^2-kb^2}{(1-ka)^2+k^2b^2}$ and $(1-ka)^2+k^2b^2>0$ since  if $b = 0$ then $a \neq 1/k$. We also have $a-ka^2-kb^2=-k(a^2+b^2)+a$ and so for $k >\tfrac{(1/\sqrt{3})\vert b\vert+|a|}{a^2+b^2}\ge\tfrac{|a|}{a^2+b^2}$, it follows that $\operatorname{Re} z<0$. We note as well that for $k>\tfrac{|a|}{a^2+b^2}$, $k(a^2+b^2)-a)$ is positive and increasing, as a function of $k$, and that $\tfrac{(1/\sqrt{3})\vert b\vert+|a|}{a^2+b^2}\ge\tfrac{|a|}{a^2+b^2}$.  We now compute the ratio of the imaginary and real part of $z$:
\begin{eqnarray*}
\left\vert\dfrac{\operatorname{Im} z}{\operatorname{Re} z}\right\vert&=& \left\vert\frac{b}{k(a^2+b^2)-a}\right\vert\\
&< & \dfrac{\vert b\vert}{\left\vert \left(\tfrac{(1/\sqrt{3})\vert b\vert+|a|}{a^2+b^2}\right)(a^2+b^2)-a\right\vert}\\
&=& \dfrac{\vert b\vert}{(1/\sqrt{3})\vert b\vert+|a|-a}\\
&\leq& \sqrt{3}.\\
\end{eqnarray*}

The result now follows from Corollary \ref{corSLC}.

\end{proof}
\begin{cor}
Every graph $G$ on $n$ vertices is an induced subgraph of a very well-covered graph $H$ such that the sequence of coefficients of $I(H,x)$ is unimodal.
\end{cor}

\section{Concluding Remarks}

While Theorem~\ref{thmexistsk} shows that for any graph $G$, the independence polynomial of some iterated leafy extension of $G$ is unimodal, the question remains as to whether this is true for {\it every} iterated leafy extension, and, in particular, for the leafy extension of $G$.  From the properties of linear fractional transformations, we can explicitly state where the independence roots of $G$ need to lie to ensure its leafy extension has a log concave (and hence unimodal)  independence polynomial.

\begin{theorem}\label{thmrootsregion}
If the roots of $I(G,x)$ lie outside of the region bounded by the union of circles with with radii $\tfrac{\sqrt{3}}{3}$ centred at $\tfrac{1}{2}+\tfrac{\sqrt{3}i}{6}$ and $\tfrac{1}{2}-\tfrac{\sqrt{3}i}{6}$, then the $i(G^*,x)$ is strictly log-concave.
\end{theorem}
\begin{proof}
We will find the image of the region $R=\{z\in\Complex:\tfrac{2\pi}{3}<\vert\text{arg}(z)\vert<\tfrac{4\pi}{3}\}$ under the M{\"o}bius transformation $f(z)=\tfrac{z}{1+z}$. Such a transformation sends lines and circles to lines and circles, and interiors/exteriors of circles and half-planes are sent to the same set of regions. We need only find the image of three points on the two line segments bounding the sector. The images of $-1+\sqrt{3}i$, $0$, and $\infty$ are $1+\tfrac{\sqrt{3}i}{3}$, $0$, and $1$ respectively, yielding the circle $C_1$, centred at $\tfrac{1}{2}+\tfrac{\sqrt{3}i}{6}$ with radius $\tfrac{\sqrt{3}}{3}$. As $1/2$ is inside $C_1$ and gets mapped to $\tfrac{1}{3}$, which is above the line $\text{arg}(z)=\tfrac{2\pi}{3}$, the {\it exterior} of $C_1$ gets mapped below the line $\text{arg}(z)=\tfrac{2\pi}{3}$.

Similarly, the images of $-1-\sqrt{3}i$, $0$, and $-\infty$ are $1-\tfrac{\sqrt{3}i}{3}$, $0$, and $1$ respectively, yielding the circle $C_2$, centred at $\tfrac{1}{2}-\tfrac{\sqrt{3}i}{6}$ with radius $\tfrac{\sqrt{3}}{3}$. As $1/2$ is inside $C_2$ and is mapped to $\tfrac{1}{3}$ which is below the line $\text{arg}(z)=\tfrac{4\pi}{3}$, the exterior of $C_2$ gets mapped below the line $\text{arg}(z)=\tfrac{4\pi}{3}$. Therefore, if we take an exterior point to the union of $C_1$ and $C_2$ is must have image under $f$ above the line $\text{arg}(z)=\tfrac{4\pi}{3}$ and below the line $\text{arg}(z)=\tfrac{2\pi}{3}$, i.e., in the region $R$. Therefore, by Corollary \ref{corSLC}, $I(G^*,x)$ is strictly log-concave.
\end{proof}

\begin{figure}[htb]
\begin{center}
\includegraphics[scale=0.4]{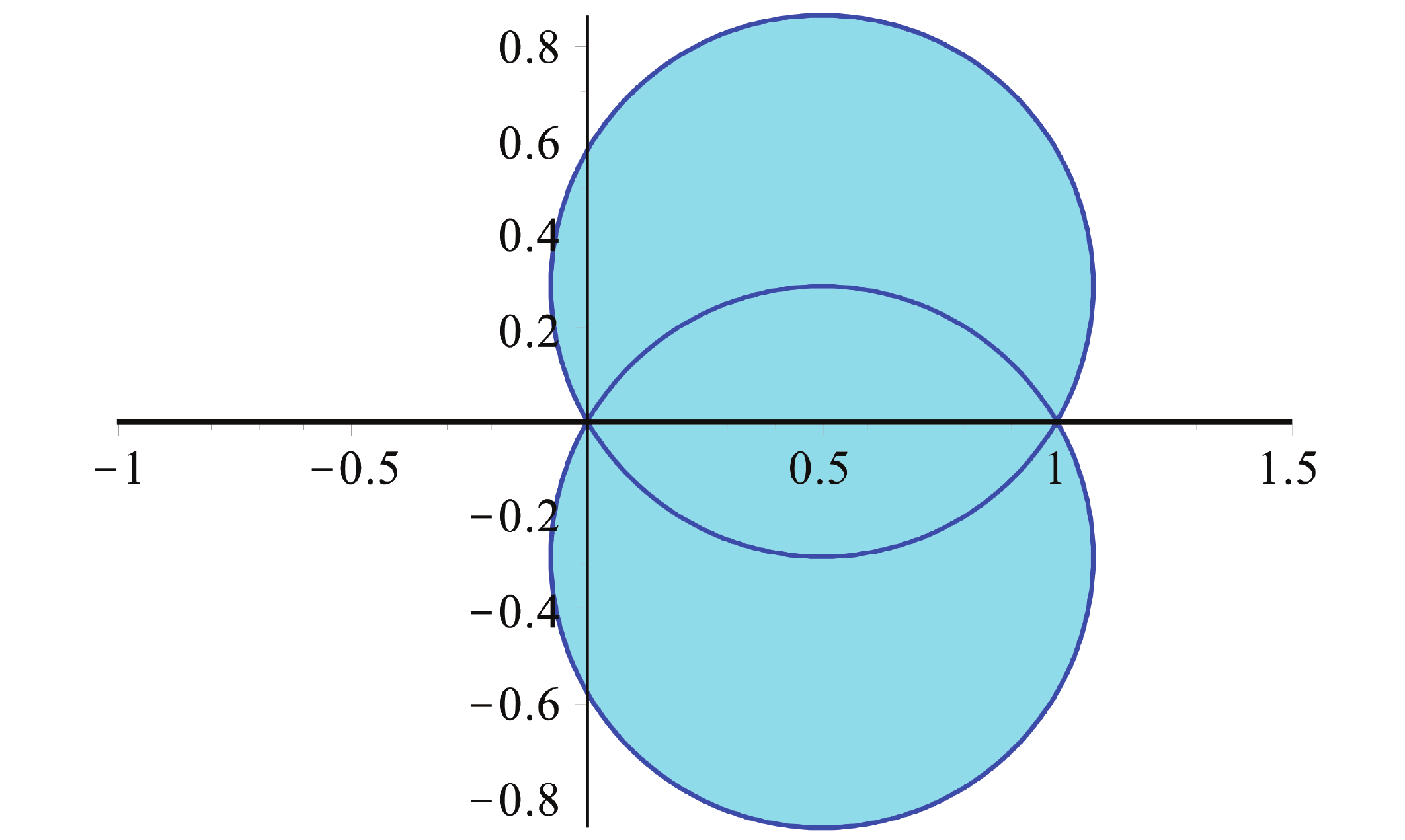}
\caption{Region that ends up outside the sector $R=\{z\in\Complex:\tfrac{2\pi}{3}\le\vert\text{arg}(z)\vert<\tfrac{4\pi}{3}\}$ under the M{\"o}bius transformation $f(z)=\tfrac{z}{1+z}$.}
\label{Region}
\end{center}
\end{figure}

Theorem \ref{thmrootsregion} assures us that as long as the roots of $I(G,x)$ are outside of a region in $\Complex$ with area a little more than $2$, then $I(G^\ast,x)$ will be strictly log-concave. Although this result works for many graphs, Brown, Hickman, and Nowakowski \cite{INDROOTS} showed that the set of independence roots of all graphs is dense in $\Complex$ even when restricted to well-covered graphs. Therefore, there exist graphs with independence roots in the union of the interior of the two circles specified in the statement of Theorem \ref{thmrootsregion} and therefore graphs that have leafy extensions with independence roots outside of $R$. In fact, using the methods outlined in \cite{INDROOTS} to find independence roots throughout $\Complex$, we have found that $I(G^\ast,x)$ has a root outside of $R$ for $G=P_5[\overline{K_{6}}]$ and $G=P_6[\overline{K_{11}}]$ (here $G[H]$ is the graph formed from $G$ by substituting a copy of $H$ in for each vertex of $G$ -- this is sometimes known as the lexicographic product of $G$ with $H$), although the independence polynomials in these cases turn out to be log-concave as well. 

Another point to note is that there exist graphs that are very well-covered but are not the leafy extension of another graph: some examples which were already pointed out are the bipartite graphs $K_{n,n}$ among others. Our results do not encompass these results; however, Finbow et al. \cite{nowawcgirth} showed that, with the exceptions of $K_1$ and $C_7$, a graph $G$ with $\operatorname{girth}(G)\ge 6$ is well-covered if and only if its pendant edges form a perfect matching. It is easy to see that the pendant edges of $G$ forming a perfect matching is equivalent to $G=H^*$ for some graph $H$ and therefore, our results apply to every well-covered graph with girth at least $6$.

\begin{figure}[htb]
\begin{center}
\includegraphics[scale=0.4]{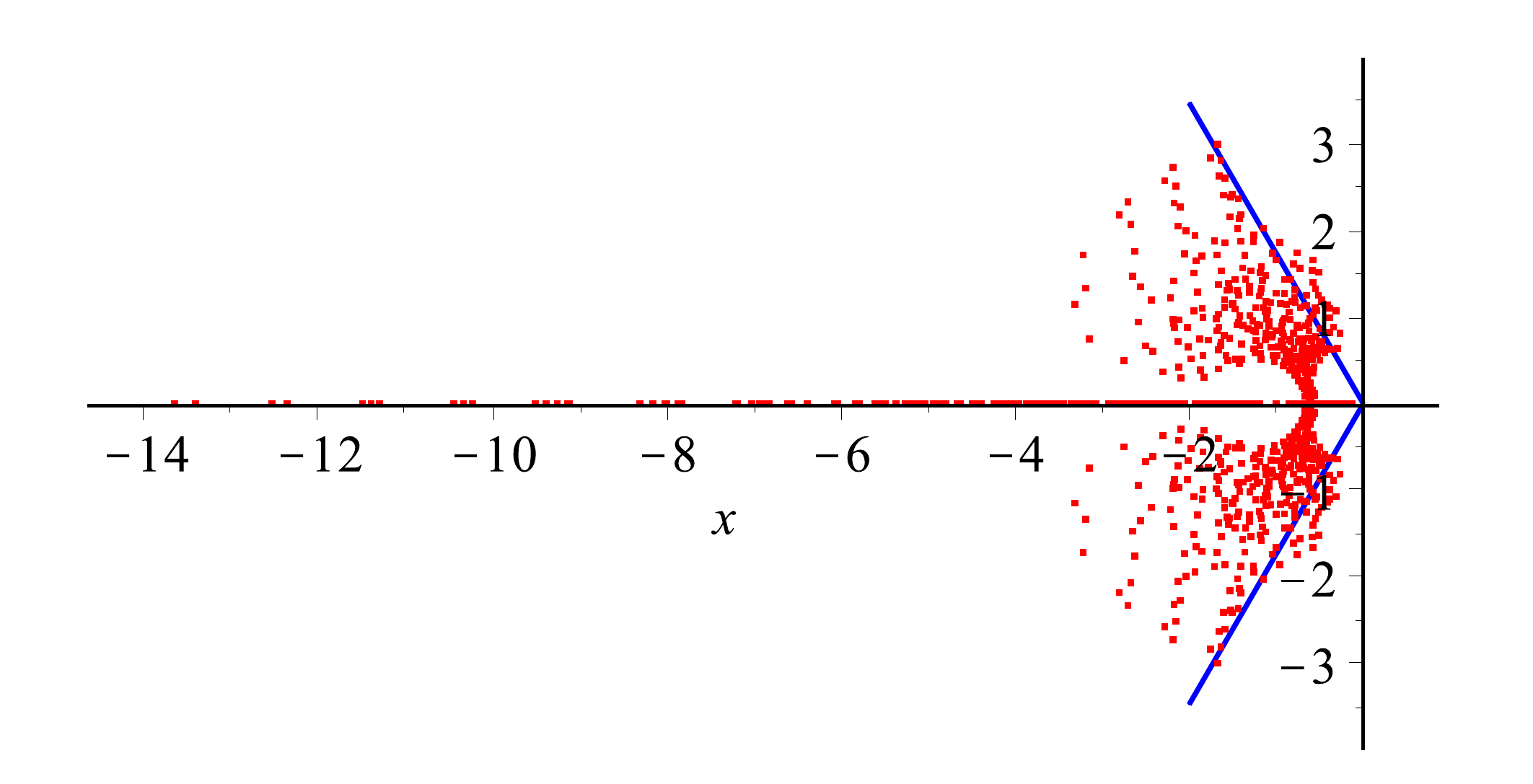}
\caption{Independence roots of all connected graphs of order $8$.}
\label{order8}
\end{center}
\end{figure}

\begin{figure}[htb]
\begin{center}
\includegraphics[scale=0.4]{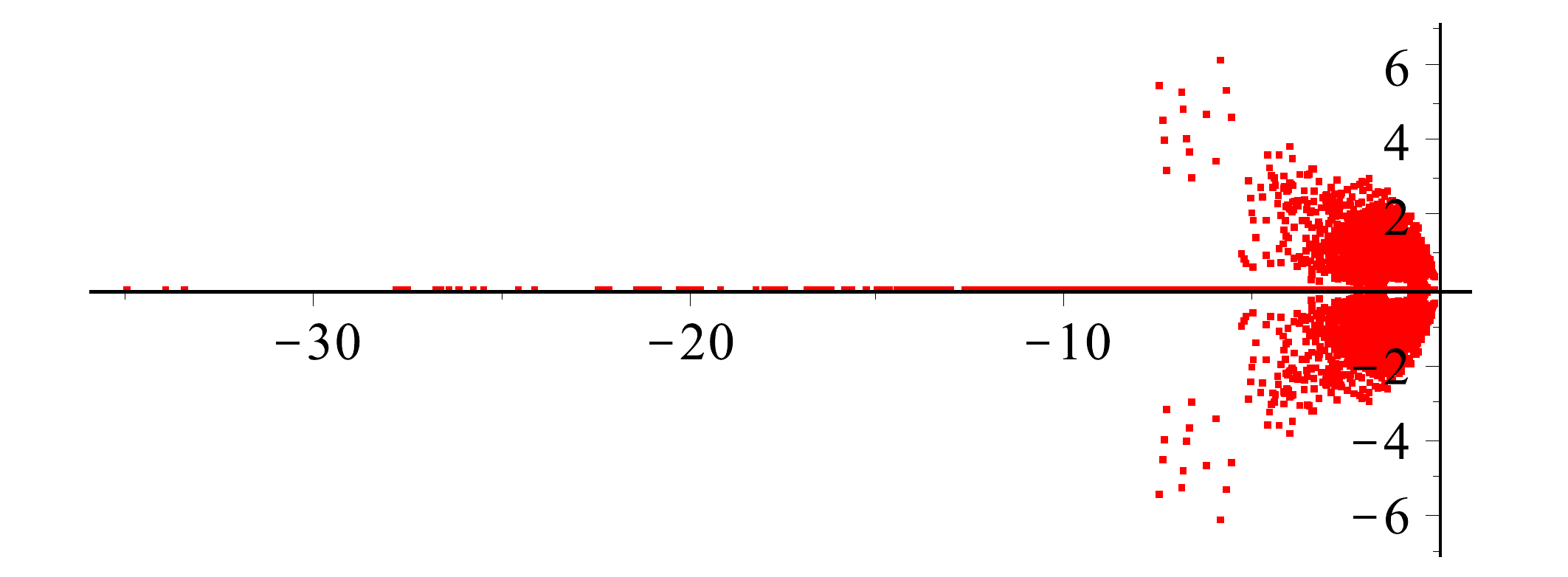}
\caption{Independence roots of all trees of order $14$.}
\label{trees14}
\end{center}
\end{figure}

Finally, the distribution of independence roots with respect to the sector  
\[ \{ z\in\mathbb{C}:\tfrac{2\pi}{3}<|\text{arg}(z)|< \tfrac{4\pi}{3}\}\]
is a fascinating one. Some computations suggest that most small graphs have their roots in the sector; in fact, out of $11,117$ connected graphs of order $8$, there are only $40$ independence roots (counting multiplicities) outside the sector (see Figure~\ref{order8}). Hence we propose the following:

\begin{conj}
The independence polynomial of almost every graph $G$ of order $n$ has all of its roots in the sector $\{ z\in\mathbb{C}:\tfrac{2\pi}{3}<|\text{arg}(z)|< \tfrac{4\pi}{3}\}$, and hence is log concave. 
\end{conj}

The same seems to be true of trees as well, and in fact we have not been able to find a single independence root of a tree in the right half-plane (see Figure~\ref{trees14}). so we also conjecture the following:

\begin{conj}
The independence polynomials of trees $G$ are {\it stable}, that is, their roots lie in the left half-plane. 
\end{conj}

\vskip0.2in
\noindent {\bf \large Acknowledgements:} The first author would like to acknowledge the support of the Natural Sciences and Engineering Research Council of Canada (grant number $170450$--$2013$). 


\bibliographystyle{plain}
\bibliography{VWC}

\begin{thebibliography}{10}

\bibitem{ALAVI}
Y.~Alavi, P.~J. Malde, A.~J. Schwenk, and P.~Erd{\"o}s.
\newblock {The vertex sequence of a graph is not constrained}.
\newblock {\em Congressus Numer.}, 1987.

\bibitem{Andrews2000}
G.~E.; Andrews, R.; Askey, and R.~Roy.
\newblock {\em {Special functions, Encyclopedia of Mathematics and its
  Applications 7}}.
\newblock Cambridge University Press, 2000.

\bibitem{Arocha1984}
J.~L. Arocha.
\newblock {Propriedades del polinomio independiente de un grafo}.
\newblock {\em Revista Ciencias Matematicas}, V:103--110, 1984.

\bibitem{BRENTISURVEY}
F.~Brenti.
\newblock {Log-concave and unimodal sequences in algebra, combinatorics, and
  geometry: an update, in Jerusalem Combinatorics 93}.
\newblock {\em Contemporary Mathematics}, 178:71--89, 1994.

\bibitem{SLCREGION}
F.~Brenti, G.~F. Royle, and D.~G. Wagner.
\newblock {Location of zeros of chromatic and related polynomials of graphs}.
\newblock {\em Canad. J. Math}, 46:55--80, 1994.

\bibitem{brownwc}
J.~I. Brown, K.~Dilcher, and R.~J. Nowakowski.
\newblock {Roots of independence polynomials of well-covered graphs}.
\newblock {\em J. Algebraic Combin.}, 11:197--210, 2000.

\bibitem{INDROOTS}
J.~I. Brown, C.~A. Hickman, and R.~J. Nowakowski.
\newblock {On the location of the roots of independence polynomials}.
\newblock {\em Journal of Algebraic Combinatorics}, 19:273--282, 2004.

\bibitem{BROWNAVERAGE}
J.~I. Brown and R.~J. Nowakowski.
\newblock {Avergae Independence Polynomials}.
\newblock {\em J. Comb. Th. B}, 93:313--318, 2005.

\bibitem{Chen}
S.-Y. Chen and H.-J. Wang.
\newblock {Unimodality of independence polynomials of very well-covered
  graphs}.
\newblock {\em Ars Combin.}, 97A:509--529, 2010.

\bibitem{Chudnovsky2007}
Maria Chudnovsky and Paul Seymour.
\newblock {The roots of the independence polynomial of a clawfree graph}.
\newblock {\em Journal of Combinatorial Theory, Series B}, 97(3):350--357,
  2007.

\bibitem{colbook}
C.J. Colbourn.
\newblock {\em {The Combinatorics of Network Reliability}}.
\newblock Oxford University Press, Oxford and New York, 1987.

\bibitem{COMTET}
L.~Comtet.
\newblock {\em {Advanced Combinatorics}}.
\newblock Reidel, Boston, 1974.

\bibitem{nowawcgirth}
A.; Finbow, B;~Hartnell, and R.~J. Nowakowski.
\newblock {A characterization of well covered graphs of girth 5 or greater}.
\newblock {\em J. Combin. Th. B}, 57:44--68, 1993.

\bibitem{Gutman1992}
I.~Gutman.
\newblock {Some relations for the independence and matching polynomials and
  their chemical applications}.
\newblock {\em Bul. Acad. Serbe Sci. arts}, 105:39--49, 1992.

\bibitem{INDFIRST}
I.~Gutman and F.~Harary.
\newblock {Generalizations of the matching polynomial}.
\newblock {\em Utilitas Mathematica}, 24:97--106, 1983.

\bibitem{Gutman1995}
I.~Gutman and X.~Li.
\newblock {A unified approach to the first derivatives of graph polynomials}.
\newblock {\em Discrete Applied Mathematics}, 58:293--297., 1995.

\bibitem{HEILMANN}
O.~J. Heilmann and E.~H. Lieb.
\newblock {Theory of monomer-dimer systems}.
\newblock {\em Communications in Mathematical Physics}, 25(3):190--232, 1972.

\bibitem{Hoede1994}
C.~Hoede and X.~Li.
\newblock {Clique polynomials and independent set polynomials of graphs}.
\newblock {\em Discrete Mathematics}, 125:219--228, 1994.

\bibitem{chromhuh}
June Huh.
\newblock {Milnor numbers of projective hypersurfaces and the chromatic
  polynomial of a graph}.
\newblock {\em J. Amer. Math. Soc.}, 25:907--927, 2012.

\bibitem{relhuh}
June Huh.
\newblock {h-Vectors of matroids and logarithmic concavity}.
\newblock {\em Advances in mathematics}, 270:49--59, 2015.

\bibitem{Levit2002}
V.~E. Levit and Eugen Mandrescu.
\newblock {On well-covered trees with unimodal independence polynomials}.
\newblock {\em Congressus Numerantium}, 159:193--202, 2002.

\bibitem{VWCOVEREDLOGCONCAVE}
V.~E. Levit and Eugen Mandrescu.
\newblock {Very well-covered graphs with log-concave independence polynomials}.
\newblock {\em Carpathian Journal of Mathematics}, 20:73--80, 2004.

\bibitem{INDPOLY}
V.~E. Levit and Eugen Mandrescu.
\newblock {The independence polynomial of a graph-a survey}.
\newblock {\em Proceedings of the 1st International Conference on Algebraic
  Informatics}, V:233--254, 2005.

\bibitem{levitgeneric}
V.~E. Levit and Eugen Mandrescu.
\newblock {Independence polynomials of well-covered graphs: generic
  counterexamples for the unimodality conjecture}.
\newblock {\em Europ. J. Combin}, 27:931--939, 2006.

\bibitem{LEVITBOOK}
V.~E. Levit and Eugen Mandrescu.
\newblock {Independence polynomials and the unimodality conjecture for very
  well-covered, quasi-regularizable, and perfect graphs}.
\newblock In {\em Graph Theory in Paris Proceedings of a Conference in Memory
  of Claude Berge}, pages 243--254, 2007.

\bibitem{TRAVES}
T.~S. Michael and W.~N. Traves.
\newblock {Independence sequences of well-covered graphs: non-unimodality and
  the roller-coaster conjecture}.
\newblock {\em Graphs and Combin.}, 19:403--411, 2003.

\bibitem{read}
R.~C. Read.
\newblock {An introduction to chromatic polynomials}.
\newblock {\em Journal of Combinatorial Theory}, 4:52--71, 1968.

\bibitem{STANELYSURVEY}
R.~P. Stanley.
\newblock {Log-concave and unimodal sequences in algebra, combinatorics, and
  geometry}.
\newblock {\em Annals of the New York Academy of Sciences}, 576:500--535, 1989.

\end{thebibliography}

\end{document}